\documentclass[reqno]{amsart}

\usepackage[utf8]{inputenc}
\usepackage[T1]{fontenc}

\usepackage{hyperref}
\usepackage{amsmath}
\usepackage{amssymb}
\usepackage{amsfonts}
\usepackage{graphicx}
\usepackage{amsthm}
\usepackage{enumerate}
\usepackage{lscape}
\usepackage{dsfont}
\usepackage{color}
\usepackage{mathtools}

\usepackage{setspace}
\newcommand{\Q}{\mathbb{Q}}
\newcommand{\R}{\mathds{R}}                   
\newcommand{\Z}{\mathds{Z}}

\newcommand{\CP}{\mathds{C}\mathrm{P}}

\newcommand{\CH}{\mathds{C}\mathrm{H}}

\newcommand{\C}{\mathds{C}}            
\newcommand{\de}{\partial}          

\newcommand{\K}{K\"{a}hler}

\newcommand{\OO}{{\mathcal{O}}}

\newcommand{\ov}[1]{\overline{#1}}
\newcommand{\wi}[1]{\widetilde{#1}}

\newcommand{\deb}{\ov\partial}

\newcommand{\ngh}{neighbourhood}
\newcommand{\lmb}{\lambda}

\newcommand{\W}{\Omega}

\newcommand{\va}{\varphi}
\newcommand{\hs}{\hspace{0.1em}}

\newcommand{\FR}{{\mathfrak{R}}}
\newcommand{\FF}{{\mathfrak{F}}}
\newcommand{\CG}{{\mathcal{G}}}

\newcommand{\vf}{\varphi}

\newcommand{\mc}{{\mathcal C}}

\newcommand{\nash}{Nash algebraic holomorphic function}

\def\F{\mathcal F}

\newtheorem{theor}{Theorem}[section]

\newtheorem{lem}[theor]{Lemma}
\newtheorem{cor}{Corollary}

\newtheorem{remark}{Remark}

\begin{document}

\title[Rigidity properties of homogeneous K\"{a}hler manifolds ]{Rigidity properties of holomorphic isometries into  homogeneous K\"{a}hler manifolds}

\author{Andrea Loi}
\address{(Andrea Loi) Dipartimento di Matematica \\
         Universit\`a di Cagliari (Italy)}
         \email{loi@unica.it}

\author{Roberto Mossa}
\address{(Roberto Mossa) Dipartimento di Matematica \\
         Universit\`a di Cagliari (Italy)}
         \email{roberto.mossa@unica.it}

\thanks{
The authors are supported by INdAM and  GNSAGA - Gruppo Nazionale per le Strutture Algebriche, Geometriche e le loro Applicazioni, by GOACT - Funded by Fondazione di Sardegna and 
partially funded by PNRR e.INS Ecosystem of Innovation for Next Generation Sardinia (CUP F53C22000430001, codice MUR ECS00000038).}


\begin{abstract}
We prove two  rigidity  results  on holomorphic isometries into  homogeneous \K\ manifolds.
The first  shows that a \K-Ricci soliton induced by the homogeneous metric of  the \K\ product of a special generalized flag manifold (i.e. a flag of classical type or integral type)  with a bounded homogeneous domain is trivial, i.e. \K-Einstein.
In the second one we prove that: (i)
a flat space is  not relative to the \K\ product  of a special generalized flag manifold   with a homogeneous bounded domain, (ii)
a special generalized flag manifold  is  not relative to the \K\ product  of a flat space with a homogeneous bounded domain
and  (iii)  a homogeneous bounded domain  is  not relative to the \K\ product  of  a flat space with a special generalized  flag manifold.
Our  theorems strongly  extend  the results   in   \cite{Cheng2021}, \cite{CDY},   \cite{LMpams}, \cite{PRIMO}  and \cite{UmearaC}.
\end{abstract}
 
\maketitle

\tableofcontents  

\section{Introduction}

 It is an interesting and classical  problem to study rigidity properties  of  holomorphic isometries (i.e. \K\ immersions) from a   \K\ manifold $(M, g)$ into another \K\ manifold $(S, g_S)$, namely those holomorphic maps
 $\varphi: M\rightarrow S$ such that $\varphi^*g_S=g$. The word {\em rigidity} has acquired several  different meanings in the mathematical literature. One can ideally subdivide the known rigidity phenomena  into three main directions, closely related to each other.
 The first  direction  tends to analyze when a  local holomorphic isometry $\varphi: U\rightarrow S$ extends to $M$ and its  unicity.
 For example Calabi   in his celebrated work \cite{Cal} (see also \cite{LoiZedda-book} and \cite{YUANYUAN}) proves that 
 a holomorphic isometry of a real analytic  \K\ manifold $(M, g)$ into a finite or infinite dimensional complex space form  $(S, g_S)$ is unique up to rigid motions (see also \cite{green} where it is shown that this is generically  true  for arbitrary real analytic \K\ manifolds) and the global extendability of a   local holomorphic isometry when $M$ is simply-connected. An interesting result in this  direction is due to Huang and Yuan in  \cite{HY} (see also \cite{HYbook}) where 
 they study some algebro-geometric conditions for the extendibility of  local holomorphic isometries between product of  Hermitian symmetric spaces of compact type. Regarding the unicity of holomorphic isometries between bounded symmetric domains  most of the work have been done by Mok  (the reader is referred  to the survey paper of Mok \cite{MOK}  and reference therein).

 The second direction deals with the following general question: {\em if the \K\ metric  $g_S$ on $S$  is homogeneous,
what can be said on the  \K\ metric $g$ on $M$  if one assumes that it is {\em canonical} in a broad sense, e.g.  it is \K-Einstein (KE) or more generally it has constant scalar curvature or it is extremal or  it is a  \K-Ricci soliton (KRS)?}

The case of holomorphic isometries of a \K-Einstein manifold $(M, g)$ into a finite or infinite dimensional complex space form  has attracted the interest of many mathematicians 
by becoming a subject on its own right (the reader is referred to Chapter 4  in \cite{LoiZedda-book}  for an update material on the subject).  For the case of holomorphic isometries of a complex manifold equipped with an  extremal metric
 into finite or infinite dimensional 
complex space forms the reader is referred to \cite{LSZext}.

Regarding   KRS  the authors of the present paper prove  the following (see also \cite{LSZkrs}):

\vskip 0.3cm

\noindent
{\bf Theorem A} (\cite{LMpams} and \cite[(i) of Theorem 1.1]{PRIMO})
{\em  If  $(g, X)$ is a KRS on a complex manifold $M$ ($g$ is a \K\ metric and $X$ is the solitonic vector field) 
and $(M, g)$ admits a holomorphic isometry into either  a finite  dimensional (definite or indefinite) complex space form or a  homogeneous bounded  domain,   then the soliton is trivial, i.e. $g$ is KE.}

\vskip 0.3cm

 Recall that a  homogeneous bounded domain is a complex  bounded domain $\Omega\subset \C^n$ equipped with a homogeneous 
\K\ metric $g_\Omega$. 
Notice that  bounded symmetric domains with the Bergman metric
are  a very special examples  of homogeneous bounded domains and 
that it could exist many  homogeneous  \K\ metrics  on a  bounded  domain   different from the Bergman metric (see \cite{SIL} for details).

The  third direction studies the obstruction of  the existence of  holomorphic isometries of a \K\ manifold $(M, g)$ (of positive dimension)  into two given \K\ manifolds $(S, g_S)$ and $(S', g_{S'})$.
If two such holomorphic isometries  exist,  $(S, g_S)$ and $(S', g_{S'})$ are said to be {\em relatives} (\cite{UmearaC} and \cite{diloi}).
For update results on relatives \K\  manifolds  the reader is referred  the survey paper \cite{YUANYUAN} 
(see also the recent preprint \cite{LMblowup}). Recently the authors of this paper  extend the result in \cite{Cheng2021} for bounded symmetric domains to arbirtary
homogeneous domain, by proving the following:

\vskip 0.3cm

\noindent
{\bf Theorem B} (\cite[(ii) of Theorem 1.1]{PRIMO})
{\em A homogeneous bounded domain and the (definite or indefinite) complex Euclidean space are not relatives.}

\vskip 0.3cm

In this paper we address the problem of   extending  Theorem A and Theorem B  to arbitrary  homogeneous \K\ manifolds,
i.e. those  \K\ manifolds which are acted upon transitively by their biholomorphic isometry group. 

Besides the homogeneous bounded domains, important examples of homogeneous \K\ manifolds are the flat ones and the generalized flag manifolds, 
the later being compact and simply-connected \K\  homogeneous manifolds.
In this paper  we  restrict to the class of  {\em special generalized flag manifolds} defined as follows. 

\vskip 0.3cm

{\bf Definition.}
{\em A generalized  flag manifold $(\mathcal C, g_{\mathcal C})$ is said to be {\em special} if it   satisfies one of the following conditions.
\begin{itemize}
\item[(i)]
${\mathcal C}$ is of classical type;
\item[(ii)]
${\mathcal C}$ is  of integral type, 
namely
there exists  a positive real number $\lambda\in\R^+$ such that $\lambda\omega_{\mathcal C}$ is integral 
(i.e. $\lambda\omega_{\mathcal C} \in H^2({\mathcal C}, \Z)$),
where  $\omega_{\mathcal C}$ denotes the \K\ form  associated to $g_{\mathcal C}$.
\end{itemize}  }

\vskip 0.3cm

All the generalized flag manifolds with second Betti number equals to one are examples of generalized flag manifold of integral type.
In particular all the Hermitian symmetric spaces of compact type with any homogeneous \K\ metric are generalized flag manifolds of integral type.
Other examples of generalized  flag manifolds of integral type are the KE ones
since in this case the homogeneous \K\ form is integral being  the first Chern class of the anticanonical bundle.
Finally, one can construct examples of generalized flag manifolds of classical type but not of integral type  by taking  the \K\  product 
$(\mathcal C\times \mathcal C, \sqrt{2}g_{\mathcal C}\times g_{\mathcal C})$, where $(\mathcal C, g_{\mathcal C})$ is a generalized flag manifold
of  both classical and  integral type.

The following theorems (and their corollaries) represent our main  results.

\begin{theor}\label{mainteor1}
Let $(M, g)$ be a \K\ manifold which admits a holomorphic isometry into the \K\ product 
${\mathcal C}\times \Omega$ 
of a special generalized flag manifold ${\mathcal C}$
and a homogeneous bounded domain $\Omega$.
Then any  KRS on $(M, g)$ is trivial, i.e. $g$ is KE.
\end{theor}

By combining  this theorem  with Theorem A we immeditaly gets:

\begin{cor}\label{maincorol1}
Let  $(M, g)$ be \K\ manifold which  admits a holomorphic isometry into either a  (definite or indefinite) flat space ${\mathcal E}$, a special generalized flag 
manifold ${\mathcal C}$ or a homogeneous bounded domain $\Omega$. Then any KRS on  $(M, g)$ is trivial.
\end{cor}

\begin{remark}\rm
Theorem \ref{mainteor1}
cannot be extended to the \K\ product of a  flat space with  either a special
generalized  flag manifold 
or a homogeneous bounded domain. Indeed,  one can easily exhibit non-trivial KRS 
on $\C\times\C P^1$
(where $\CP^1$ is the complex  one-dimensional complex projective space with the Fubini-Study metric $g_{FS}$) 
and on $\C\times\C H^1$    (where $\C H^1$  is the complex one-dimensional  hyperbolic space
with the hyperbolic metric $g_{hyp}$).
Indeed, one can easily verify that the vector fleld $X$ on $\C\times \C P^1$ (resp. on $\C\times \CH^1$)
defined by  $X(z,q):= -2\left(z\frac{\de}{\de z} +  \ov z\frac{\de}{\de\ov z}\right)$ 
(resp. $X(z,q):= 2\left(z\frac{\de}{\de z} +  \ov z\frac{\de}{\de\ov z}\right)$
is a solitonic vector field but the \K\ metric $g_{0}\oplus g_{FS}$ (resp. on $g_{0}\oplus g_{hyp}$) is not KE. 
\end{remark}

\begin{remark}\rm\label{remarproj}
Notice that if  $({\mathcal C}, g_{\mathcal C})$ 
is a generalized flag manifold of integral type with $\lambda \omega_{\mathcal C}$  integral,
then $({\mathcal C}, \lambda g_{\mathcal C})$ 
admits a holomorphic isometry into a  complex projective space $(\C P^N, g_{FS})$ (see. e.g. the proof of Theorem 1 in \cite{LZUDDASGW}).
Thus, in order  to prove  Corollary \ref{maincorol1} for a  generalized flag manifold of integral type  
one can assume that this flag manifold  is $(\C P^N, g_{FS})$
and then the triviality of  the KRS in Corollary \ref{maincorol1} can be deduced   by Theorem A above.
\end{remark}

\begin{theor}\label{mainteor2} 
Let $\mathcal E$, $\mathcal C$ and $\Omega$ be respectively  a  flat manifold, a  special generalized flag manifold and a homogeneous bounded domain.
Then the following facts hold true.
\begin{itemize}
\item [(i)]
$\mathcal E$ is not relative to the \K\ product  
$\mathcal C\times \Omega$; 
\item [(ii)]
$\mathcal C$ is not relative to the \K\ product  $\mathcal E\times\Omega$; 
\item [(iii)]
$\Omega$ is not relative to the \K\ product  $\mathcal E\times \mathcal C$. 
\end{itemize}
\end{theor}

It is worth pointing out that it could exist  three \K\ manifolds $(M_1, g_1)$,  $(M_2, g_2)$  and $(M_3, g_3)$  such that 
$(M_1, g_1)$ is not relative to either  $(M_2, g_2)$ or  $(M_3, g_3)$ but it is relative to $(M_2\times M_3, g_2\oplus g_3)$. 
Therefore in Theorem  \ref{mainteor2}  one cannot restrict to a single factor of the \K\ product involved in order to achieve the conclusion.
For example the  one-dimensional complex projective spaces  $(\CP^1, g_{FS})$  and $(\CP^1, 2g_{FS})$ where $g_{FS}$
is the Fubini Study metric, are not relatives. Indeed, assume by contradiction that there exist an open set $U\subset \C$ and two holomorphic isometries $\varphi_1 :U\rightarrow \C P^1$ and $\varphi_2 :U\rightarrow \C P^1$ such that 
$\varphi_1^*g_{FS}=\varphi_2^*(2g_{FS})$. By restricting $U$ if necessary one can assume that this maps are injective and hence $\varphi=\varphi_1\circ\varphi_2^{-1}: \varphi (U)\rightarrow \CP^1$ would be 
a holomorphic map such that $\varphi^*g_{FS}=2g_{FS}$, in contrast with \cite[Theorem 13]{Cal}. 
Nevertheless, $(\C P^1\times \C P^1, g_{FS}\oplus g_{FS})$ and $(\CP^1, 2g_{FS}))$  are relatives (the holomorphic  map $\varphi: \C P^1\rightarrow \C P^1\times \C P^1, q\mapsto (q, q)$ 
satisfies $\varphi^*(g_{FS}\oplus g_{FS})=2g_{FS}$).

As a very special case of Theorem \ref{mainteor2} one gets the following appealing result which is  a strong  extension
of  the results in \cite{LMpams}, \cite{PRIMO}.

\begin{cor}\label{maincorol2}
Any two among a flat  space  ${\mathcal E}$, a special generalized flag  manifold ${\mathcal C}$ or a homogeneous bounded domain
 $\Omega$ are not relatives.
\end{cor}

\vskip 0.3cm
Since  each compact homogeneous KE  manifold is the Kahler product of a flat complex torus and generalized flag  KE manifold (see, e.g.  (\cite{Be}))
Theorem \ref{mainteor2} gives:

\begin{cor}
A  homogeneous bounded  domain is not relative to a compact KE homogeneous manifold. 
\end{cor}

The proofs of  Theorem \ref{mainteor1} and of (i) of Theorem  \ref{mainteor2} are obtained by some results on KRS proved in \cite{LMpams} and \cite{PRIMO}
and on the structure of the  Calabi diastasis function of special generalized flag manifolds (Lemma \ref{lemdiastflag}) 
and of homogeneous bounded domains  (Lemma \ref{lembis}).
On the other hand the proof of (ii) and (iii) of Theorem \ref{mainteor2}  and in particular  the fact that a special generalized flag manifold is
not relative  to a homogeneous bounded domain, are    based on Theorem \ref{proprelnash} below,  a new  and technical result, 
dealing   with  trascendental properties  of holomorphic Nash algebraic functions.

The paper  contains two more sections. Section \ref{secnash} is dedicated to the proof of Theorem \ref{proprelnash}
while Section \ref{proofs} to the proofs of Theorems \ref{mainteor1} and \ref{mainteor2}.

\vskip 0.3cm 
The authors would like to thank Yuan Yuan for all interesting discussions and comments on \nash s.

\section{A technical result on holomorphic Nash algebraic functions}\label{secnash}
Let $\mathcal N^m$ be the set of real analytic functions $\xi : V\subset \C^m \to \R $ defined in a \ngh\ $V\subset \C^m$ of the origin, such that its real analytic extension $\tilde\xi (z, w)$ around $(0,0)\in V \times \operatorname {Conj} V$ is a holomorphic Nash algebraic function. We define
$$
\mathcal F= \left\{\xi \left(f_1,\dots,f_m\right) \mid \xi \in \mathcal N^m,\ f_j \in \OO_0,\ f_j(0)=0,\ j=1,\dots,m, \ m>0  \right\},
$$
where $\OO_0$ denotes the germ of holomorphic functions around $0\in \C$.
We set 
\begin{equation}\label{tildeF}
\wi {\mathcal F}= \left\{\xi \left(f_1,\dots,f_m\right) \in \mathcal F \mid \xi  \text{ is of diastasis-type},\ m>0 \right\}.
\end{equation}

We say (see also \cite{LMpams}) that  a real analytic function defined on a neighborhood   $U$ of   a point $p$
of a complex manifold $M$ is of  {\em diastasis-type} if 
in one (and hence any) coordinate system 
$\{z_1, \dots , z_n\}$ centered at $p$ its  expansion in $z$ and $\bar z$ 
does not contains non constant purely holomorphic or anti-holomorphic terms (i.e. of the form $z^{j}$ or $\bar{z}^{j}$ with  $j > 0$). 

Before proving Theorem \ref{proprelnash} we need a lemma.

\begin{lem}\label{lemnashexpr} Let $f_0, f_1,\dots, f_s$ be non zero \nash s, such that
\begin{equation}\label{eqnashexpr00}
f_0^{c_0}f_1^{c_1}\dots f_s^{c_s}=1, \qquad c_0, \dots, c_s \in \R.
\end{equation}
If  $\left\{c_0, \dots, c_s\right\}$ are linearly independent over $\Q$, then each $f_j$, $j=0,\dots ,s$ is a  constant function.
\end{lem}

\begin{proof}
Assume for example  that $f_0$ is not constant.  Let $z_0\in \C \cup \left\{\infty\right\}$ such that $\lim_{z\to z_0} f_0(z)=\infty$. 
 Without loss of generality we can assume $z_0\in \C$ otherwise we apply the following argument to the  \nash\ $\frac{1}{f_0}$.
 Moreover, up to a translation, we can also assume  that $z_0=0$.  
 
 By the Puiseux expansion we can write
$
f_\ell(z)=(1+o(1))d_\ell\hs z^{k_\ell/N_\ell}, 
$
where $d_\ell\in\C\setminus\left\{0\right\}$, $\frac{k_\ell}{N_\ell}\in \Q_{<0}$, $\ell=1,\ldots,s$.
 From \eqref{eqnashexpr00} we have that $\frac{k_0 c_0}{N_0}+\dots+\frac{k_s c_s}{N_s}=0$. Since $\lim_{z\to z_0} f_0(z)=\infty$, we see that $k_0$ and at least one of $k_1, \dots, k_s$ are not zero in contrast 
 with the linearly independence of $\left\{c_0, \dots, c_s \right\}$ over $\Q$.
\end{proof}

The following theorem is  the technical result needed  in the proof  of (ii) and (iii) of Theorem \ref{mainteor2}. It is a  generalization of \cite[Theorem 2.1 (iii)]{CDY} (even for $s=1$). 

\begin{theor}\label{proprelnash}
Let $\psi_\ell=\xi_\ell\left(f_{\ell,1},\dots,f_{\ell,m_\ell}\right) \in \wi {\mathcal F}$, $\ell=0,1,\dots,s$, be such that
\begin{equation}\label{eqnashexpr}
\psi_1^{c_1}\dots \psi_s^{c_s}=\psi_0, \quad  \psi_0(0)\neq 0,  \quad c_1, \dots, c_s \in \R.
\end{equation}
If  $\left\{c_1, \dots, c_s, 1\right\}$ are linearly independent over $\Q$, then $\psi_1, \dots, \psi_s$ are constant functions.
\end{theor}
\begin{proof}

Let us write 
\begin{equation}\label{egz}
\psi_k=\xi _k\left(f_{1}^{(k)},\dots,f_{m_k}^{(k)}\right)\in \F,\  k=0,\dots,s.
\end{equation}
We rename the functions involved  
\begin{equation}\label{eqfvarf}
\left(\varphi_1, \dots , \varphi_t\right)=\left(f_{1}^{(0)},\dots,f_{m_0}^{(0)},\dots, f_1^{(s)},\dots,f^{(s)}_{m_s}\right)
\end{equation}
and set
$$
S=\left\{\varphi_1, \dots , \varphi_t\right\}.
$$ 

Let $D$ be an open neighborhood  of the origin of $\C$ on which each $\varphi_j$,  $j=1, \dots ,t$, is defined.
Consider   the field  $\mathfrak R$ of rational function on $D$ and its  field extension 
$
\mathfrak F = \mathfrak R \left(S\right),
$
namely, the smallest subfield of the field of the  meromorphic functions on $D$, containing  rational functions and the elements of $S$. Let $l$ be the transcendence degree  of the field extension $\FF/\FR$. 
If $l=0$, then each element in $S$ is holomorphic  Nash algebraic, hence $\psi_1, \dots, \psi_s$ would be constant functions by Lemma \ref{lemnashexpr}. Assume then that $l>0$.
Without loss of generality we can assume that  $\CG=\left\{\varphi_1,\dots,\varphi_l\right\}\subset S$  is a maximal algebraic independent subset over $\FR$. Then   there exist minimal polynomials $P_j\left(z, X,Y\right)$, $X=\left(X_1,\dots,X_l\right)$,  such that 
$$
P_j\left(z,\Phi(z), \va_j(z)\right)\equiv 0, \ \forall j=1, \dots ,t,
$$
where $\Phi(z)=\left(\va_1(z),\dots,\va_l(z)\right)$.

Moreover, by  the definition of minimal polynomial 
 $$
 \frac{\de P_j\left(z,X,Y\right)}{\de Y}\left(z,\Phi (z),\va_j(z)\right)\not\equiv 0, \ \forall j=1, \dots ,t.
 $$
  on $D$.
 Thus, by the algebraic version of the existence and uniqueness part of the implicit function theorem, there exist a connected open subset $U\subset D$ with $0\in \ov U$ and Nash algebraic functions 
 $\hat \va_j(z,X)$,  defined in a neighborhood $\hat U$ of $\left\{(z, \Phi(z)) \mid z \in U \right\}\subset \C^n \times \C^{l}$, such that
$$
\va_j(z)=\hat \va_j\left(z,\Phi(z)\right), \ \forall j=1, \dots ,t.
$$
for any $z\in U$.

Let us denote
$$
\left(\hat f_{1}^{(0)}(z, X),\dots,\hat f_{m_0}^{(0)}(z, X),\dots,\hat f_{1}^{(s)}(z, X),\dots,\hat f_{m_s}^{(s)}(z, X)\right)=(\hat\varphi_1 (z, X), \dots , \hat\varphi_t(z, X)),
$$ 
(notice that $\hat f_{i}^{(k)}(z, \Phi(z))= f_{i}^{(k)}(z)$ for all $k=0, \dots ,s$ and $i=1, \dots ,m_k$).
We define
$$
\hat F_k(z,X):=\left(\hat f_{1}^{(k)}(z,X),\dots,\hat f_{m_k}^{(k)}(z,X)\right),\ k=0, \dots ,s.
$$

Consider the function 
\begin{equation*}
 \Psi(z,X,w):=
\tilde\xi _0\left(\hat F_0(z,X), F_0(w) \right)\\
\tilde\xi _1\left(\hat F_1(z,X), F_1(w)\right)^{-c_1}\cdots\left(\tilde\xi _s\hat F_s(z,X), F_s(w)\right)^{-c_s},
\end{equation*}
where $\tilde\xi _j$ is the real analytic extension of $\xi _j$ in a \ngh\ of $(0,0) \in \C^{m_j} \times \operatorname {Conj} \C^{m_j}$ and $F_k(w)=\left(f_{1}^{(k)}(w),\dots,f_{m_k}^{(k)}(w)\right)$. 
By shrinking $U$ if necessary we can assume that $\Psi(z,X,w)$  is defined on  $\hat U \times U$.

We claim that $\Psi(z,X,w)$ is identically equal to one on this set.
Recalling that $\psi _k(z,w)=\tilde \xi _k \left(F_k(z),F_k(w)\right)$ is of diastasis-type, we see that 
$\tilde \xi _k (\hat F_k(z,X),F_k(0))=\psi _k(0)$,
 $k=0,\dots,s$. Since $0\in \ov U$, it follows by \eqref{eqnashexpr}
that  $\Psi(z,X,0)\equiv 1$.
 Hence, in order to prove the claim, it is enough to show that the logarithmic differentiation with respect $w$,
 namely 
 \begin{equation*}\begin{split} 
(\de^{\log}_w \Psi)(z,X,w)&=\frac{\de_w \tilde\xi _0\left(\hat F_0(z,X), F_0(w) \right)}{\tilde\xi _0\left(\hat F_0(z,X), F_0(w) \right)}\\
&- c_1 \frac{\de_w \tilde\xi _1\left(\hat F_1(z,X), F_1(w) \right)}{\tilde\xi _1\left(\hat F_1(z,X), F_1(w) \right)} - \dots - c_s \frac{\de_w \tilde\xi _s\left(\hat F_s(z,X), F_s(w) \right)}{\tilde\xi _s\left(\hat F_s(z,X), F_s(w) \right)}
\end{split}\end{equation*}
vanishes  for all $w\in U$.  Assume, by contradiction, that there exists $w_0\in U$ such that $(\de^{\log}_w \Psi)(z,X,w_0)\not\equiv 0$. Since $(\de^{\log}_w \Psi)(z,X,w_0)$ is Nash algebraic in $(z,X)$ there exists a holomorphic polynomial $P(z,X,t)=A_d(z,X)t^d+\dots+A_0(z,X)$ with $A_0(z,X)\not\equiv 0$ such that  $P(z,X,(\de^{\log}_w \Psi)(z,X,w_0))=0$. 
Since,  by \eqref{eqnashexpr}  and  \eqref{egz} we have
 $\Psi(z,\Phi (z),w)\equiv 1$, we get $(\de^{\log}_w \Psi)(z,\Phi (z),w)\equiv 0$. Thus  $A_0(z,\Phi (z))\equiv 0$
 which contradicts  the fact that $\va_1(z),\dots,\va_l(z)$ are algebraic independent over $\FR$. Hence $(\de^{\log}_w \Psi)(z,X,w)\equiv 0$ and the claim is proved.

Therefore
\begin{equation*}\begin{split} 
{\tilde\xi _0\left(\hat F_0(z,X), F_0(w) \right)} = 
\left(\tilde\xi _1\left(\hat F_1(z,X), F_1(w)\right)\right)^{c_1}\cdots \left(\tilde\xi _s\left(\hat F_s(z,X), F_s(w)\right)\right)^{c_s},
\end{split}\end{equation*}
for every $(z,X,w)\in \hat U \times U$. By fixing  $w\in U$ and  applying Lemma \ref{lemnashexpr} we deduce that $\psi_1, \dots, \psi_s$ are constant functions. The proof of the theorem  is complete.
\end{proof}

\section{The proofs of Theorem \ref{mainteor1} and Theorem \ref{mainteor2}}\label{proofs}
Given a   real analytic 
\K \ metric $g$  on  a complex manifold 
 $M$ one can introduce
in a neighborhood of a point 
$p\in M$,
a very special
\K\ potential
$D^g_p$ for the metric
$g$, 
the celebrated
{\em Calabi's diastasis function} (see \cite{Cal} or to \cite{LoiZedda-book}
for details). 
Among all the potentials the diastasis
is characterized by the fact that 
in every coordinate system 
$\{z_1, \dots , z_n\}$ centered at $p$
\begin{equation*}\label{eqcardi}
D^g_p(z, \bar z)=\sum _{|j|, |k|\geq 0}
a_{jk}z^j\bar z^k,
\end{equation*}
with 
$a_{j 0}=a_{0 j}=0$
for all multi-indices
$j$. Clearly the diastasis $D^g_p$ is  a function of diastasis-type as defined above.

The following two lemmata are crucial in both the  proofs of Theorem \ref{mainteor1} and Theorem \ref{mainteor2}.

 \begin{lem}\label{lemdiastflag}
Let $\mathcal C$ be a generalized  flag manifold of special type and $g_{\mathcal C}$ its homogeneous metric. Then around any point $p \in \mc$ there exists a system of coordinates, centered in the origin, such that the diastasis  $D^{g_{\mathcal C}}_0$ satisfies 
\begin{equation}\label {eqdiasflagf}
e^{D^{g_\mathcal C}_0}\in \wi{F}^{c_1}\cdots \wi{F}^{c_s}, \qquad c_1,\dots,c_s\in \R^+. 
\end{equation}
\end{lem}
\begin{proof}
Assume that $\mathcal C$ is of classical type.  Given $p\in {\mathcal C}$ by  \cite[Theorem 3.5]{lmzlog},  there exists a system of Bochner coordinates 
$(z_1,\dots,z_n)$
centered at $p$ and dense in $\mathcal C$
such that the diastasis associated of  $g_\mc$  in a neighborhhood of $p$ reads as 
\begin{equation}\label{eqdiastflag0}
D^{g_\mathcal C}_0(z)=\sum_{j=1}^s c_j \log\left(h_j(z) \right)
\end{equation}
where  $c_1, \dots, c_s \in \R^+$  and $h_1, \dots , h_s$ are real analytic functions  (see \cite[Theorem 3.5]{lmzlog} for their explicit expression), the \K\ potential \eqref{eqdiastflag0} has been firstly constructed in \cite[Proposition 8.2]{AlPer} then in \cite[Theorem 3.5]{lmzlog} it has been  showed that it is a diastasis function. 
In \cite[pag. 9]{lmzlog} (see also \cite{lmzboch}) it is proven that the functions $h_1, \dots, h_s$ are polynomials.  Thus  \eqref{eqdiasflagf}
holds  true and by  \cite{AlPer} we see that $g_{\mathcal C}$ is integer if and only if $c_1, \dots, c_s \in \Z^+$.  Moreover,  
for any choice of $c_1, \dots, c_s \in \R^+$ in \eqref{eqdiastflag0} we obtain the diastasis function of a homogeneous metric on $\mc$. 
Assume  now that $\mathcal C$ is of integral type, so that there exists a \K\ immersion $F:\mc \to \C P^N$ such that $g_\mc =\lmb F^*g_{FS}$, for some $\lambda\in{\R^+}$ (cfr. Remark \ref{remarproj}).
Fix a system of coordinates $z_1,\dots,z_n$ for $\mc$ around $p$. From the hereditary property of the diastasis function, we get
$$
D^{g_\mathcal C}_0(z)=\lmb\log \left(1 + \left\|F(z)\right\|^2\right),
$$
where $\left\|F(z)\right\|^2=|f_1(z)|^2+\dots+|f_N(z)|^2 \in \wi{F}$ and $f_1, \dots, f_N$ are the component of $F$ written in the  affine coordinates of $\C P^N$ around $f(0)$. Clearly \eqref{eqdiasflagf} is satisfied with $s=1$ and $c_1=\lmb$ and $\omega_\mc$ is integral if and only if $c_1 \in \Z^+$. 
\end{proof}

\begin{lem}\label{lembis}\emph{(\cite[Proof of Theorem 1.1]{PRIMO})}
Let  $\left(\Omega, g_\Omega\right)$ be a homogeneous bounded domain. Consider its realization as a Siegel domain of $\C^n$, then the diastasis function for the metric $g_\Omega$  at given point $v\in\Omega$ is 
given by:
\begin{equation}\label{diastomega}
D^{g_\Omega}_v(u)=\sum_{k=1}^r\gamma_k
\log\left(\frac{F_{k}(u,\ov u)\hs F_{k}(v,\ov v)}{F_{k}(u,\ov v)\hs F_{k}(v,\ov u)}\right),
\end{equation}
where $\gamma_k$ are positive real numbers and  $F_k$ are non costant rational holomorphic functions, $k=1, \dots ,r$.  
In particular
\begin{equation}\label{eqdiashomf}
e^{D^{g_\Omega}_0(u)} \in \wi F^{\gamma_1} \cdots \wi F^{\gamma_r}.
\end{equation}
\end{lem}

We are now in the position to prove Theorem \ref{mainteor1} and Theorem \ref{mainteor2}.

\begin{proof}[Proof of Theorem \ref{mainteor1}]
Let  $(g, X)$ be a KRS on  $M$ and let 
$\varphi:M\rightarrow \mc\times\W$ be a  holomorphic isometric immersion, i.e. $\varphi^*(g_\mathcal C\oplus g_\W)=g$, 
where $g_\mathcal C$ and  $g_\W$  are the homogeneous \K\ metrics on $\mathcal C$ and  $\W$ respectively.
Choose complex coordinates $\{z_1, \dots , z_n\}$ for $M$ centered at $p\in M$ where Calabi's diastasis function $D_p^g$ is defined.
By the hereditary property of the diastasis functiuon, we have that
\begin{equation}\label{eqdipr}
D^{g}_p(z)=D^{g_\mc}_{\varphi_\mc (p)}(\vf_\mc(z))+D^{g_\W}_{\varphi_\W (p)}(\vf_\W(z)),
\end{equation}
on a \ngh\ of $p$ and 
where $\varphi:=\left(\vf_\mathcal C,\vf_\W\right)$. From Lemma \ref{lemdiastflag} and \ref{lembis} we see that 
$$
e^{D^{g}_p(z)} \in \wi F^{c_1}\cdots \wi F^{c_s} \wi F^{\gamma_1}\cdots \wi F^{\gamma_r}.
$$
Thus, by \cite[Proposition  4.1]{PRIMO},  the metric $g$ is forced to be KE.
\end{proof}

\begin{proof}[Proof of (i) of Theorem \ref{mainteor2}]
Assume by contradiction that   $\mathcal E$ and $\mc\times\W$ are relatives.  Thus there exist  a \ngh\  $U \subset \C$ of the origin $0\in \C$ 
and two holomorphic immersions $\varphi_\mathcal E:U\to \mathcal E$ and 
$\varphi:U \to \mathcal C\times\W$  such that 
\begin{equation}\label{eqfond}
 \varphi_\mathcal E^* g_0=\varphi^* \left(g_\mathcal C\oplus g_\W \right), 
\end{equation}
 where $g_\mathcal C$ and  $g_\W$  are the homogeneous \K\ metrics on $\mathcal C$ and  $\W$ respectively and $g_\mathcal E$ is the flat metric on $\mathcal E$.  
 Since we are working locally we can assume that $\mathcal E$ is the complex Euclidean space $\C^m$, $g_\mathcal E$ is the flat metric $g_0$ on $\C^m$ and that $\varphi_\mathcal E(0)=0\in\C^m$, $m=\dim \mathcal E$.
 
 Write $\varphi:=(\varphi_\mathcal C,\varphi_\W)$ with  $\varphi_\mathcal C:U\rightarrow \mathcal C$ and $\varphi_\W: U\rightarrow\W$. Again from the hereditary property of the diastasis function we get
 $$
 D^{g_0}_{\varphi_\mathcal E(0)}(\varphi_\mathcal E(z)) = D^{g_\mc}_{\varphi_\mc(0)}(\varphi_\mc(z)) + D^{g_\W}_{\varphi_\W(0)}(\varphi_\W(z))
 $$  
 From Lemma \ref{lemdiastflag} and \ref{lembis} we see that 
$$
e^{D^{g_0}_{\varphi_\mathcal E(0)}(z)} \in \wi F^{c_1}\cdots \wi F^{c_s} \wi F^{\gamma_1}\cdots \wi F^{\gamma_r}.
$$
Thus,   by the trascendental properties of holomorphic Nash algebraic functions  proved in  \cite[Theorem 2.1]{PRIMO} we deduce  that $\varphi_\mathcal E$ is constant, in constrast with the hypothesis that $\varphi_\mathcal E$ is an immersion. The proof of (i) is completed.
\end{proof}

\begin{proof}[Proof of (ii) and (iii) of Theorem \ref{mainteor2}]
We start by proving (ii) and (iii) in the special case ${\mathcal E}$ is zero dimensional, namely we prove that ${\mathcal C}$
and $\Omega$ are not relatives. As we have already pointed out at in the introduction its proof 
strongly  relies on Theorem \ref{proprelnash}.

Let $U\subset \C$ be a \ngh\ of the origin and assume by contradiction that there exists two holomorphic immersions
$\varphi_\mathcal C:U \to \mathcal C$ 
and $\varphi_\W:U\to  \W$ such that 
\begin{equation}\label{eqbsthm12a}
\varphi_\mc^*g_\mc=\varphi_\W ^* g_\W,
\end{equation}
with $\varphi_\mathcal C:U\to \mathcal C $ and $\varphi_\W:U\to \W$. So that in a \ngh\ of $0\in \C$ we have 
$$
{D^{g_\mathcal C}_{\varphi_\mc(0)}(\varphi_\mc(z))}= {D^{g_\W}_{\varphi_\W(0)}(\varphi_\W(z))}.
$$
Let us pass to the coordinates centered at $\varphi_\mc(0)$ for $\mc$ given by Lemma \ref{lemdiastflag} and to the realization of $\W$ as a Siegel domain. 
From Lemma \ref{lemdiastflag} and Lemma \ref{lembis}  combined with \eqref{eqdiastflag0} and  \eqref{diastomega} we see that in a \ngh\ of $0\in \C$ there exist $\psi_1, \dots, \psi_{s}$, $\phi_{1}, \dots, \phi_{r}\in \wi F$ such that
\begin{equation}\label{meglio1}
{D_{\va_\mathcal C (0)}^{g_\mathcal C}(\va_\mathcal C (z))}=\log\left(\psi_1^{c_1} \cdots \psi_{s}^{c_s} \right)
\end{equation}
and
\begin{equation}\label{meglio2}
{D_{\va_\W (0)}^{g_\W}(\va_\W (z))}=\log\left(\phi_{1}^{\gamma_1}, \dots, \phi_{r}^{\gamma_r}\right).
\end{equation}
Up to  indexes order, we can assume that there exist $1\leq s'\leq s$, such that  $\left\{c_1, \dots, c_{s'}\right\}\subset \left\{c_1, \dots, c_{s}\right\}$ are maximal subsets of linearly independent real numbers over $\Q$. Thus for  suitable  coefficients $a_{jk}\in \Q$, we have 
\begin{equation}\label{eqcoefa}
c_{j}=\sum_{k=1}^{s'}a_{jk}c_k, \quad j=s'+1,\dots, s.
\end{equation}
Set 
\begin{equation}\label{eqtildepsi}
\tilde\psi_k:=\psi_k \psi_{s'\!+\!1}^{a_{s'\!+\!1,k}}\cdot\ldots\cdot \psi_{s}^{a_{sk}}, \quad k=1,\dots s'.
\end{equation}
Therefore we have
\begin{equation}\label{eqdias=1}
1= e^{D_{\va_\mathcal C (0)}^{g_\mathcal C}(\va_\mathcal C (z))-D_{\va_\W (0)}^{g_\W}(\va_\W (z))}
 =\tilde\psi_1^{c_1}\cdot \ldots \cdot \tilde\psi_{s'}^{c_{s'}} \cdot \phi_{1}^{-\gamma_1}\cdot \ldots \cdot \phi_{r}^{-\gamma_{r}}.
\end{equation}
Let us complete $\left\{c_1, \dots, c_{s'}\right\}$ to a maximal subset 
$$
\left\{c_1, \dots, c_{s'}, \gamma_1, \dots, \gamma_{r'}\right\}\subset \left\{c_1, \dots, c_{s'}, \gamma_1, \dots, \gamma_{r}\right\},\qquad 0\leq r' \leq r,
$$ 
of linearly independent real numbers over $\Q$
(with $r'=0$ we mean that $\left\{c_1, \dots, c_{s'}\right\}$ is already maximal). 
So that
$$
\gamma_j=\sum_{k=1}^{s'}b_{jk}c_k+\sum_{k=1}^{r'}t_{jk}\gamma_k,\qquad j=r'+1,\ldots,r 
$$
we can rewrite \eqref{eqdias=1} as follows
\begin{align*}
 1&=\left(\tilde\psi_1 \phi_{r'+1}^{-b_{r'+1, 1}}\cdots \phi_{r}^{-b_{r1}}\right)^{c_1}\dots \left(\tilde\psi_{s'} \phi_{r'+1}^{-b_{r'+1, {s'}}}\cdots \phi_{r}^{-b_{rs'}}\right)^{c_{s'}}\\
&\ \cdot\left(\phi_1 \phi_{r'+1}^{t_{r'+1, 1}}\cdots \phi_{r}^{t_{r1}}\right)^{-\gamma_1}\cdots\left(\phi_{r'} \phi_{r'+1}^{t_{r'+1, {r'}}}\cdots \phi_{r}^{t_{rr'}}\right)^{-\gamma_{r'}}
\end{align*}
By applying Theorem \ref{proprelnash}, we conclude that 
\begin{equation}\label{tildepsik1}
\tilde\psi_{k} =A_k\phi_{r'+1}^{b_{r'+1,k}}\cdot\ldots\cdot\phi_{r}^{b_{rk}}, \quad A_k \in \R, \quad k=1,\dots, s'.  
\end{equation}
and
\begin{equation}\label{tildepsik2}
\phi_{k} =B_k\phi_{r'+1}^{-t_{r'+1,k}}\cdot\ldots\cdot\phi_{r}^{-t_{rk}}, \quad B_k \in \R, \quad k=1,\dots, r'.  
\end{equation}
Let us choose $\tilde c_1,  \dots, \tilde c_{s'}\in \Z^+$ and $\tilde\gamma_1, \dots,\tilde \gamma_{r'}\in \R^+$ such that
\begin{equation}\label{tildegammaj}
\tilde\gamma_j=\sum_{k=1}^{s'}b_{jk}\tilde c_k+\sum_{k=1}^{r'}t_{jk}\tilde \gamma_k > 0,\qquad j=r'+1,\ldots,r  
\end{equation}
and
\begin{equation}\label{tildecj}
\tilde c _{j}:={a_{j 1}\tilde c_1+\dots+a_{j s'}\tilde c_{s'} }\in\Z^+, \quad j=s'+1,\dots, s,
\end{equation}
where the $a_{jk}$ are the coefficients appearing in \eqref{eqcoefa}.
Then  \eqref{tildepsik1}, \eqref{tildepsik2} and \eqref{tildegammaj} yield 
\begin{equation}\label{eqrelpsiphi}\begin{split} 
 & \tilde\psi_1^{\tilde c_1}\cdot \ldots \cdot \tilde\psi_{s'}^{\tilde c_{s'}}\cdot   \phi_1^{-\tilde \gamma_1}\cdot \ldots \cdot \phi_{r'}^{-\tilde \gamma_{r'}}\\
 =&\left(A_1\phi_{r'+1}^{b_{r'+1,1}}\cdot\ldots\cdot\phi_{r}^{b_{r1}}\right)^{\tilde c_1}
 \cdot \ldots \cdot
 \left(A_{s'}\phi_{r'+1}^{b_{r'+1,{s'}}}\cdot\ldots\cdot\phi_{r}^{b_{rs'}}\right)^{\tilde c_{s'}}\\
  \cdot&\left(B_1\phi_{r'+1}^{-t_{r'+1,1}}\cdot\ldots\cdot\phi_{r}^{-t_{r1}}\right)^{-\tilde \gamma_1}
 \cdot \ldots \cdot
 \left(B_{r'}\phi_{r'+1}^{-t_{r'+1,r'}}\cdot\ldots\cdot\phi_{r}^{-t_{rr'}}\right)^{-\tilde \gamma_{r'}}\\
=&C \phi_{r'+1}^{\tilde \gamma_{r'+1}}\cdot\ldots\cdot \phi_{r}^{\tilde \gamma_{r}},
\end{split}\end{equation}
for some $C \in \R$.
Using \eqref{eqtildepsi} and \eqref{tildecj}, we get
\begin{equation*}\begin{split} 
 \tilde\psi_1^{\tilde c_1}\cdot \ldots \cdot \tilde\psi_{s'}^{\tilde c_{s'}}
 &=\left(\psi_1 \psi_{s'+1}^{a_{s'\!+1\, 1}}\cdot\ldots\cdot \psi_{s}^{a_{s1}}\right)^{\tilde c_1}\cdot\ldots\cdot 
\left(\psi_{s'} \psi_{s'+1}^{a_{s'\!+1\, s'}}\cdot\ldots\cdot \psi_{s}^{a_{ss'}}\right)^{\tilde c_{s'}}\\
&=\psi_1^{\tilde c_1}\cdot\ldots\cdot\psi_{s'}^{\tilde c_{s'}}\cdot \psi_{s'+1}^{a_{s'\!+1\, 1}\tilde c_1+\dots+a_{s'\!+1\, s'}\tilde c_{s'} }
\cdot\ldots\cdot
\psi_{s}^{a_{s 1}\tilde c_1+\dots+a_{s s'}\tilde c_{s'} }\\
&=\psi_1^{\tilde c_1}
\cdot\ldots\cdot
\psi_s^{\tilde c_s}.
\end{split}\end{equation*}
The previous equation together with  \eqref{eqrelpsiphi}  yield
\begin{equation*}\begin{split} 
 \de\deb \log \left( 
\psi_1^{\tilde c_1}
\cdot\ldots\cdot
\psi_s^{\tilde c_s}
\right)
=\de\deb \log \left(
\tilde\psi_1^{\tilde c_1}\cdot \ldots \cdot \tilde\psi_{s'}^{\tilde c_{s'}}
\right)= \de\deb \log\left(\phi_{1}^{\tilde \gamma_1}\cdot\ldots\cdot \phi_{r}^{\tilde \gamma_{r}}\right).
\end{split}\end{equation*}

From this equation we deduce 
that $\mc$ equipped with the  homogeneous metric $\tilde g_\mc$  whose diastasis is 
obtained by replacing in \eqref{meglio1} the constants $c_1, \dots, c_s$   with  the positive integers $\tilde c_1,\dots, \tilde c_s$
and  $\W$ equipped with the metric $\tilde g_\W$  whose diastasis is 
obtained by replacing in  \eqref{meglio2}  $\gamma_1, \dots, \gamma_r$ with  $\tilde \gamma_1, \dots, \tilde  \gamma_r$,
are relatives. Since the coefficients $\tilde c_1,\dots, \tilde c_s$ are positive integers it follows  that the metric  $\tilde g_\mc$ is projectively induced (cfr. Remark \ref{remarproj} above) contradicting  \cite{Mossa} where it is shown  that a homogeneous bounded  domain and a projective manifold cannot be  relatives.

We can now prove (ii) of Theorem \ref{mainteor2} (the proof of (iii)  is omitted since it follows the same pattern).
Assume by contradiction that there exist two holomorphic maps
$\varphi_\mathcal C:U \to \mathcal C$ 
and $\varphi=\left(\varphi_\mathcal E,\varphi_\W \right):U\to \mathcal E \times  \W$ 
such that 
$$
\varphi_\mc^*g_\mc=\varphi ^*(g_\mathcal E \oplus g_\W), 
$$
where $U \subset \C$ is a \ngh\ of the origin.
By arguing as in proof of  part (i), we see that the previous equation yields 
$$
e^{D^{g_0}_{\varphi_\mathcal E(0)}(z)} \in \wi F^{c_1}\cdots \wi F^{c_s} \wi F^{-\gamma_1}\cdots \wi F^{-\gamma_r}.
$$
Again by  \cite[Theorem 2.1]{PRIMO} we see that $\varphi_\mathcal E$ is constant,  and so 
$\varphi_\mc ^*g_\mc = \varphi^*_\W g_\W$, 
showing that ${\mathcal C}$ and $\Omega$ are relatives, in contrast with the first part of the proof.
\end{proof}

\end{document}